\documentclass[final]{svjour3}
\smartqed

\usepackage{bm, amsmath,amssymb}
\usepackage{graphicx}
\usepackage{tikz}

\newcommand{\an}[1]{\langle #1 \rangle}

\newcommand\enorm[1]{\left|\!\left|\!\left| #1\right|\!\right|\!\right|}	
\newcommand{\ngrad}{\bm n \cdot \nabla}
\newcommand{\dK}{\partial K}
\newcommand{\rd}{\partial}

\newcommand{\Proj}{{\sf P}}
\newcommand{\T}{\mathcal{T}}
\newcommand{\E}{\mathcal{E}}
\newcommand{\G}{\mathcal{G}}
\newcommand{\meas}[1]{\mathrm{meas}(#1)}
\newcommand{\diam}{{\rm diam}}
\newcommand{\sumK}{\sum_{K \in \T_h}}
\newcommand{\I}{{\sf I}}
\newcommand{\pair}[1]{\{#1\}}

\title{A hybridized discontinuous Galerkin method with reduced stabilization }
\author{Issei Oikawa}     
\institute{Issei Oikawa \at Waseda University Faculty of 
 Science and Engineering \\
 \email{oikawa@aoni.waseda.jp}
}

\begin{document} 
\maketitle

\begin{abstract} 
 In this paper, we propose a hybridized discontinuous Galerkin(HDG) method with reduced stabilization for the Poisson equation. 
The reduce stabilization  proposed here  enables us to use piecewise polynomials of degree $k$ and $k-1$ for the approximations of element and inter-element unknowns, respectively, 
unlike the standard HDG methods.
We provide the error estimates in the energy and $L^2$ norms under the chunkiness condition.
In the case of $k=1$, it can be shown that the proposed method 
is closely related to the Crouzeix-Raviart nonconforming finite element method.
Numerical results are presented to verify the validity of the proposed  method.
\end{abstract}



\section{Introduction}
  In this paper, we propose a new hybridized discontinuous Galerkin(HDG) method with reduced stabilization.
 We consider the Poisson equation with homogeneous Dirichlet boundary condition as a model problem:
\begin{subequations} 
\begin{eqnarray}
 -\Delta u &=& f    \mbox{ in } \Omega, \label{poissoneq}\\
        u &=& 0    \mbox{ on }  \partial\Omega. \label{poissonbc}
\end{eqnarray}
\end{subequations}
Here $\Omega \subset \mathbb R^2$ is a convex polygonal domain,
 and $f \in L^2(\Omega)$ is a given function. 
For simplicity, we here deal with only the two-dimensional case, although 
the proposed method can be applied to the three-dimensional problems.
  
 The HDG methods are already applied to various problems and are still being developed. 
  For second order elliptic problems, the HDG methods were introduced and analyzed by Cockburn et al. \cite{CDG08,CGL09}.
 The embedded discontinuous Galerkin(EDG) methods, which is the continuous-approximation version of the HDG method, were analyzed in \cite{CGSS09}. In these HDG methods, the numerical traces are hybridized in a mixed formulation.  
 In \cite{OK10}, another approach of the HDG method was proposed.
 The hybridization presented in \cite{OK10} is based on the {\em hybrid displacement method} proposed by Tong for linear elasticity problems \cite{Tong70}. 
  The resulting scheme is almost equivalent to
   the IP-H \cite{CGL09}. 

 In the formulations of HDG methods,  element and hybrid unknowns are introduced.
 The element unknown can be eliminated by the hybrid unknown, which allows us to reduce 
 the number of the globally coupled degrees of freedom.  
 In all the standard HDG methods, we need to use two polynomials of equal degree for the approximations of the element and hybrid unknowns in order to achieve optimal convergences.
  The motivation of the reduced stabilization we propose is to use piecewise polynomials of degree $k$ and $k-1$ for approximations of the  element and hybrid unknowns, respectively, which we call {\em $P_k$--$P_{k-1}$ approximation}.
 In \cite{BuSt08,BuSt09}, reduced stabilization was introduced for the discontinuous Galerkin method(the RIP-method).
In \cite{Lehrenfeld2010}, Lehrenfeld first proposed a reduced HDG scheme by adjusting a numerical flux, which is equivalent to our proposed scheme. However, no error analysis was presented. 

 In \cite{CGSS09}, it is proved that the hybrid part of an EDG solution coincides with the trace of a finite element solution with the linear element.    
 Analogously, the proposed method with $P_1$--$P_0$ approximation is closely related to the Crouzeix-Raviart nonconforming finite element method \cite{CrRa73}. We prove that our approximate solution coincides with the Crouzeix-Raviart approximation at the midpoints of edges. 
 
 We provide a priori error estimates under the chunkiness condition. 
 The optimal error estimates  in the energy norm are proved.
 In terms of the $L^2$-errors, it is shown that the convergence rates are optimal when the scheme is symmetric. 
 However, for the nonsymmetric schemes, we prove only the sub-optimal estimates in the $L^2$ norm due to the lack of adjoint consistency.  
   We also provide the easy implementation by means of the Gaussian quadrature formula in the two-dimensional case. Unfortunately, such implementation is impossible in the three-dimensional case.
    
 This paper is organized as follows. 
 Sect. 2 is devoted to the preliminaries.
 In Sect. 3, we introduce reduced stabilization, and describe the proposed method.
 In Sect. 4, we provide the error estimates in the energy and $L^2$ norms under the chunkiness condition.
 In Sect. 5, numerical results are presented to verify the validity of our scheme.
 Finally, in Sect. 6, we end with a conclusion.

\section{Preliminaries and notation}

\subsection{Chunkiness condition.} \label{sec:chunki}
 Let $\{\T_h\}_h$  be a family of meshes of $\Omega$. 
 Each element $K \in \mathcal T_h$ is assumed to be a polygonal domain {\em star-shaped} with respect to a ball of which radius is $\rho_K$. 
 Let $h_K = \diam K$ and $h = \max_{K \in \T_h}  h_K$.
 We assume that the boundary $\partial K$ of $K \in \mathcal T_h$ is composed of $m$-faces and $m$ is bounded by $M$ from above independently of $h$.
 Let us denote $\E_h = \{ e \subset \partial K : K \in \T_h\}$. 
 In this paper, we assume that the family $\{\T_h\}_h$ satisfies the {\em chunkiness condition} \cite{BS08,Kikuchi12}:
 there exists a positive constant $\gamma_C$ independent of $h$ such that
\begin{equation}\label{chunkiness-t}
  \frac{h_K}{\rho_K} \le \gamma_{C} \quad  \forall K \in \T_h.
\end{equation}
 From the chunkiness condition,  a kind of cone condition follows\cite{BS08,Kikuchi12}. 
 Let $\tilde T$ be a reference triangle  with the height of $\gamma_T>0$.
 We assume that  $\tilde T$ is an isosceles triangle.
 Let $\tilde e$ denote the base of $\tilde T$.
 For each $K \in \T_h$  and $e \subset \dK$ (Figure \ref{tricond}),
let $F_{e,K}$ be an affine-linear mapping from $\tilde T$ onto $T \subset K$ such that $F_{e,K}(\tilde e) = e$ and the height of $T$ is equal to $\gamma_T h_e$  The constant $\gamma_T$ depends only on the chunkiness parameter $\gamma_C$.
 Note that there exists a constant $\gamma_E \ge 1$ such that 
\begin{equation} \label{chunkiness-e}
  \frac{h_T}{h_e} \le \gamma_E.
\end{equation}  
\begin{figure}[ht]
\center
  \begin{tikzpicture} 
    \draw[thick] (0,0) coordinate (a) -- node[below] {$\tilde e$} (3,0) coordinate (b) -- (3/2,6/5) --cycle;
    \draw (3/4+6/5*0.05,3/5-3/2*0.05) -- (3/4-6/5*0.05,3/5+3/2*0.05);
    \draw (9/4-6/5*0.05,3/5-3/2*0.05) -- (9/4+6/5*0.05,3/5+3/2*0.05);

    \draw (a) -- (0,-1);
    \draw (b) -- (3,-1);
    \draw[<->] (0,-0.75) -- node[below] {$1$} (3,-0.75);
    \draw (3,0) -- (3.5,0);
    \draw (1.5,6/5) -- (3.5,6/5);
    \draw[<->] (3.25, 0) -- node[right] {$\gamma_T$}(3.25,6/5);
    \draw (1.5,.5) node {$\tilde T$};
    \draw[->] (4, 3/5) -- node[above] {$F_{e,K}$}(5.5,3/5);
    \begin{scope}[shift={(2,0)}]
    \draw[thick] (5,0) --(4 ,1) -- (5.5,2.25) -- (6.5,1.75) -- (7.5,2) --  (7,0) -- node[below] {$e$}  (5,0);
     \draw (5,1) node {$K$};
    \draw (5,0) --  (6,4/5) -- (7,0);
    \draw (6, 1/3) node {$T$};
    
    \draw (5,0) -- (5,-0.8);
    \draw (7,0) -- (7,-0.8);
    \draw [<->] (5,-0.6) -- node[below] {$ h_e $}(7,-0.6); 
    
    \draw (6,0.8) -- (8,0.8);
    \draw (7,0) -- (8,0);
    \draw [<->] (7.8,0.8) -- node[right] {$\gamma_T h_e $}(7.8,0.0); 
    
    \end{scope}
  \end{tikzpicture}
  \caption{Triangle condition}
  \label{tricond}
  \end{figure}
\subsection{Function spaces}
We introduce  the {\em piecewise} Sobolev spaces over $\T_h$, i.e.,
$H^k(\mathcal T_h) = \{ v \in L^2(\Omega) : v|_K \in H^k(K) \ \forall K \in \T_h \}.$   
 The {\em skeleton} of $\mathcal T_h$ is defined by $\Gamma_h = \bigcup_{e\in\mathcal E_h} e.$
 We use $L^2_D(\Gamma_h) = \{ \hat v \in L^2(\Gamma_h) : \hat v = 0 \textrm{ on } \partial\Omega \}$ and $\bm V = H^2(\T_h) \times L_D^2(\Gamma_h)$ for the hybridized formulation of the continuous problem.
 Let us define $\bm V(h) = \{ (v, v|_{\Gamma_h}) : v \in H^2(\Omega)\}$ $\subset H^2(\Omega) \times H^{3/2}(\Gamma_h)$, where $v|_{\Gamma_h}$ stands for the trace of $v_h$ on $\Gamma_h$. 
 We use the following inner products:  
\begin{eqnarray*}
   (u,v)_{\T_h} = \sumK \int_K uvdx,\quad 
   \an{u, v}_{\rd\T_h} = \sumK \int_{\dK}  u  v ds
\end{eqnarray*}
for  $u, v \in L^2(\Omega)$ or $L^2_D(\Gamma_h)$.
\subsection{Finite element spaces and projections}
 Let $\mathcal P^k(\mathcal T_h)$ be the function space of element-wise polynomials of degree $k$ over $\mathcal T_h$, and $\mathcal P^l(\mathcal E_h)$ be the space of edge-wise polynomials of degree $l$ over $\mathcal E_h$, where $k$ and $l$ are nonnegative integers. 
 Then we define $V_h^k = \mathcal P^k(\mathcal T_h)$ and $\hat V_h^l = \mathcal P^l(\mathcal E_h) \cap L^2_D(\Gamma_h)$.
 We employ $\bm V^{k,l}_h = V^k_h \times \hat V^l_h$ as finite element spaces of $\bm V$.  Let us denote by $\Proj_{k}$ the $L^2$-projection from $L^2(\Gamma_h)$ onto $\mathcal P^{k}(\E_h)$.

\subsection{Mesh-dependent norms}
 Let $\|\cdot\|_m$ and $|\cdot|_m$ be the usual Sobolev norms and seminorms in the sense of \cite{AF03}, respectively.
 We introduce auxiliary mesh-dependent seminorms:
\begin{align}
   &|v|_{1,h}^2 = \sum_{K\in\mathcal T_h} |v|_{1,K}^2 
    \quad \textrm{for} \  v \in H^1(\mathcal T_h), \\ 
   & |v|_{2,h}^2 = \sum_{K\in\mathcal T_h} h_K^2|v|_{2,K}^2 
    \quad \textrm{for} \  v \in H^2(\mathcal T_h), \\ 
   &|\bm v|_{\textrm j}^2 = \sum_{K\in\mathcal T_h} \sum_{e \subset \partial K} 
       \frac{1}{h_e}\left\| \Proj_{k-1}(\hat v -  v)\right\|^2_{0,e} 
    \quad \textrm{for} \ \bm v=\pair{v,\hat v} \in \bm V,
    \label{jump}
\end{align}
where $h_e$ is the diameter of $e$.
 Note that  $\Proj_{k-1} v$  in \eqref{jump} is defined by $\Proj_{k-1}({\rm trace}(v|_K))$, which is well-defined, whereas $v$ may be double-valued on $\dK$.
 In our error analysis, we  use the following energy norm:
\begin{eqnarray*}
&&  \enorm{\bm v}^2 = |v|_{1,h}^2 +|v|_{2,h}^2 +|\bm v|_{\textrm j}^2
 \quad \text{ for } \bm v = \pair{v,\hat v} \in \bm V.
\end{eqnarray*}

\subsection{Trace and inverse inequalities}
 We here state the trace and inverse inequalities without proofs.
 The constants appearing in the inequalities are independent of $h$, $K \in \T_h$
and $e \subset \dK$ under the chunkiness condition.
\begin{lemma}[Trace inequality]  \label{trace}
 Let $K \in \mathcal T_h$ and $e$ be an edge of $K$. 
 There exists a constant $C$ independent of $K$, $e$ and $h$ such that
    \begin{eqnarray}  \label{trace-ineq}
       \|v\|_{0,e} \le Ch_e^{-1/2} \left(\|v\|_{0,K}^2 + h_K^2|v|_{1,K}^2 \right)^{1/2}
   \qquad \forall v \in H^1(K).
    \end{eqnarray}  
\end{lemma}
\begin{proof}
 Refer to \cite{Kikuchi12}. \qed
\end{proof}

\begin{lemma}[Inverse inequality]  \label{inverse}
 Let $K \in \mathcal T_h$.
 There exists a constant $C$ independent of $K$ and $h$ such that
\begin{eqnarray}  \label{inverse-ineq}
   |v_h|_{1,K} \le Ch_K^{-1} \|v_h\|_{0,K} \qquad \forall v_h \in \mathcal P^k(K).
\end{eqnarray}  
\end{lemma}
\begin{proof}
 Refer to \cite{BS08}. \qed
\end{proof}
We will use the lemma below to bound the terms of the complementary projection, $\I-\Proj_{k-1}$. 
\begin{lemma} \label{lem:hdgnorm}
 There exists a constant $C$ independent of $h$ such that,  for all $v \in H^1(K)$,
\begin{equation} \label{eq:hdgnorm2}
  h_e^{-1}\|(\I - \Proj_{k-1})v\|_{0,e}^2 \le C |v|_{1,K}^2.
  \end{equation}
\end{lemma}
\begin{proof} 
 Let $\tilde T$ be a reference triangle and $\tilde e$ be the base of $\tilde T$ as illustrated in Figure \ref{tricond}.
 Let $\tilde v \in \mathcal P^k(\tilde T)$ be arbitrarily fixed.
 We define a linear functional on $H^1(\tilde T)$ by
\[
     G(\tilde w) = \an{ (\I-\Proj_{k-1})\tilde v, \tilde w }_{\tilde e}.
\]
 Note that the functional $G$ vanishes on $\mathcal P^{0}(\tilde T)$.
 By the Schwarz and trace inequalities, we have
\begin{eqnarray*}
   |G(\tilde w)| 
     &\le& \|(\I-\Proj_{k-1})\tilde v\|_{0, \tilde e} \|\tilde w\|_{0,\tilde e} \\
     &\le& \|(\I-\Proj_{k-1})\tilde v\|_{0, \tilde e} 
           \cdot  C h_{\tilde e}^{-1/2}(\|\tilde w\|_{0,\tilde T}^2+h_{\tilde T}^2|\tilde w|_{1,\tilde T}^2)^{1/2} \\
     &\le& C  \|(\I-\Proj_{k-1})\tilde v\|_{0, \tilde e} \|\tilde w\|_{1,\tilde T}.
\end{eqnarray*}
 By the Bramble-Hilbert lemma, we have
\[
   |G(\tilde w)| \le C\|(\I-\Proj_{k-1})\tilde v\|_{0, \tilde e} |\tilde w|_{1,\tilde T}.
\]  
 Taking $\tilde w = \tilde v$ gives us 
\begin{equation}\label{eq02}
      \|(\I-\Proj_{k-1})\tilde v\|_{0, \tilde e} \le C |\tilde v|_{1,\tilde T}.
\end{equation}
 Let $F(\tilde{\bm x})=B \tilde{\bm x} + \bm d$ be an affine mapping from $\tilde T$ onto $T \subset K$ such that  $F(\tilde e) = e$. 
 Choosing $\tilde v =v \circ F$, we have
\begin{equation}\label{eq03}
    \|(\I-\Proj_{k-1})\tilde v\|_{0, \tilde e} 
      = \frac{\meas{\tilde e}^{1/2}}{\meas{e}^{1/2}} \| (\I-\Proj_{k-1})v\|_{0,e},
\end{equation}
 where $\meas{e}$ is the measure of $e$. 
 From  \cite[Theorem 3.1.2.]{Ciarlet78}, it follows that
\begin{align}
   |\tilde v|_{1,\tilde T}
   &\le C \|B\| |\det B|^{-1/2}  |v|_{1, T}  \notag\\ 
   &\le C\frac{h_T}{2\rho_{\tilde T}} \frac{\meas{\tilde T}^{1/2}}{\meas{T}^{1/2}}
     |v|_{1, T},   \label{eq04}
\end{align}
 where $\rho_{\tilde T}$ is the radius of the inscribed ball of $\tilde T$. 
 The measure of $T$ is given by 
\begin{equation}\label{eq05}
      \meas{T} = \frac{\gamma_T h_e}{2} \meas{e}.
\end{equation}
 From \eqref{eq04}, \eqref{eq05} and \eqref{chunkiness-e}, we have
\begin{align}
   |\tilde v|_{1,\tilde T} 
      &\le C\frac{h_T}{h_e^{1/2}\meas{e}^{1/2}} |v|_{1, T} \notag \\ 
      &\le C\gamma_E \frac{h_e^{1/2}}{\meas{e}^{1/2}}|v|_{1, T}. \label{eq06}
\end{align} 
 From \eqref{eq02}, \eqref{eq03} and \eqref{eq06}, it follows that
\begin{eqnarray*}
       \| (\I-\Proj_{k-1})v\|_{0,e} 
        &\le& C h_e^{1/2} |v|_{1, T} 
        \le C h_e^{1/2} |v|_{1, K},
\end{eqnarray*}
 which completes the proof.
   \qed
\end{proof}

\subsection{Approximation property}
The approximation property in the energy norm follows   
from those of $\mathcal P^k(\T_h)$ and $\mathcal P^k(\E_h)$.  
\begin{lemma}[Approximation property] \label{lem:proj}
   Let  $v \in H^{k+1}(\Omega)$ and $\bm v = \pair{ v , v|_{\Gamma_h}}$. 
   We assume that the finite element space for the 
   hybrid unknown is discontinuous. Then
  there exists a positive constant $C$ independent of $h$ such that 
  \begin{eqnarray} \label{ap}
  \inf_{\bm v_h \in \bm V_h^{k,k-1}} \enorm{\bm v -  \bm v_h} \le C h^k |v|_{k+1}.
  \end{eqnarray}
  \end{lemma}
\begin{proof}
 It is known  that there exists $\bm w_h =\pair{w_h, \hat w_h} \in \bm V_h^{k,k}$ such that
\begin{align}  
   &  \|v - w_h \|_{0}  \le C h^{k+1} |v|_{k+1}, \quad \label{ap1a}  \\
   &  |v - w_h |_{i,h} \le C h^{k+1-i} |v|_{k+1} \quad (i = 1,2), \label{ap1b} \\
  &   \left(\sum_{K \in \T_h} 
     \sum_{e\subset \dK} h_e^{-1} \| v - \hat w_h\|^2_{0,e}\right)^{1/2}
     \le Ch^{k}|v|_{k+1}. \label{ap1c} 
\end{align}
    Let us define $\bm v_h = \{ w_h, \Proj_{k-1}\hat w_h\} \in \bm V_h^{k,k-1}$.
 Note that 
 $|\bm v - \bm v_h|_{\rm j} 
     = |\bm v_h|_{\rm j}
     = |\bm w_h|_{\rm j}$.
 By the trace inequality for $K \in \T_h$ and $e \subset \dK$, we have
\begin{align*}
    \|\Proj_{k-1}(\hat w_h - w_h)\|_{0,e}
    & \le  \|\hat w_h - w_h\|_{0,e}\\
    &\le \|\hat w_h- v \|_{0,e} 
    + \|v- w_h\|_{0,e} \\
    & \le \|\hat w_h- v \|_{0,e} 
    + Ch_e^{-1/2}(\|v- w_h\|_{0,K}^2+h_K^2|v- w_h|_{1,K}^2)^{1/2}.
\end{align*}
 From the above and \eqref{ap1a}--\eqref{ap1c}, it follows that $\enorm{\bm v - \bm v_h} \le Ch^k |v|_{k+1}$.
 \qed
\end{proof}
\paragraph{Remark.}
 In the standard HDG methods, we can impose
   the continuity at nodes on the hybrid unknowns
    to reduce the number of degrees of freedom, which is so-called  
    {\em continuous} approximation.  
     However, the approximation property with respect to the energy norm
      does not hold for the continuous approximation 
      since $\Proj_{k-1} v_h$ is not continuous at nodes
       in general even if $v_h$ is continuous.
     Indeed, it is found by numerical experiments that the convergence rates
     in the energy and $L^2$ norms are sub-optimal
      for the continuous approximation.  

\section{Reduced HDG method}
\subsection{The standard HDG scheme} 
 To begin with, we present the standard HDG formulation:
 find $\bm u_h = \pair{u_h, \hat u_h} \in \bm V_h^{k,k}$ such that 
\begin{eqnarray}
 B_{\rm std}(\bm u_h, \bm v_h)  
   = (f,v_h)_\Omega \quad \forall \bm v_h = \pair{v_h, \hat v_h} \in \bm V_h^{k,k},
\end{eqnarray}
 where the bilinear form is defined by
\begin{eqnarray}
  B_{\rm std}(\bm u_h, \bm v_h) 
   &=&      
     (\nabla u_h, \nabla v_h)_{\T_h} \label{hdg}
   + \an{\ngrad u_h, \hat v_h - v_h}_{\rd \T_h} \\ \nonumber
   &\quad&+ s\an{\ngrad v_h, \hat u_h - u_h}_{\rd\T_h} \\ \nonumber
   &\quad& +  \an{ \tau (\hat u_h - u_h), \hat v_h -  v_h}_{\rd\T_h}. 
\end{eqnarray}
 Here $s$ is a real number and $\tau$ is a stabilization parameter.
 The parameter $\tau$ takes  a constant value $\tau_e/h_e$ on each edge $e$ with $0 < \tau_0 \le \tau_e \le \tau_1 $ for some $\tau_0, \tau_1$.
 We refer to \cite{OK10} for the details of the derivation.
\subsection{Reduced HDG schemes} 
 Let us sketch  the main idea of our method.
 The second term in the convectional  scheme \eqref{hdg} can be rewritten as
\[
  \an{ \ngrad{u_h}, \hat v_h - v_h}_{\rd\T_h}
  = \an{ \ngrad{u_h}, \Proj_{k-1} (\hat v_h - v_h)}_{\rd\T_h}
\]
 since $\bm n \cdot \nabla u_h \in \mathcal P^{k-1}(\T_h)$.
 The stabilization term is correspondingly decomposed into
\begin{eqnarray*}
 \an{ \tau (\hat u_h - u_h), (\hat v_h -  v_h)}_{\rd\T_h}
 &=&\an{ \tau \Proj_{k-1} (\hat u_h - u_h), \Proj_{k-1}(\hat v_h -  v_h)}_{\rd\T_h} \\
 & &  + \an{ \tau (\I - \Proj_{k-1} )(\hat u_h - u_h), (\I - \Proj_{k-1})(\hat v_h -   v_h)}_{\rd\T_h}
\end{eqnarray*}
 Our reduced stabilization is obtained by dropping the second term in the right-hand side.
 The proposed scheme reads: 
find $\bm u_h = \pair{u_h, \hat u_h} \in \bm V_h^{k,k-1}$ such that 
\begin{eqnarray} \label{hdg-ws1}
 B_{h}(\bm u_h, \bm v_h)  
   = (f,v_h)_\Omega \quad \forall \bm v_h = \pair{v_h, \hat v_h} \in \bm V_h^{k,k-1},\end{eqnarray}
 where the bilinear form is defined by
\begin{eqnarray} \label{hdg-ws2}
  B_h(\bm u_h, \bm v_h) 
  &=&      
     (\nabla u_h, \nabla v_h)_{\T_h}
   + \an{\ngrad u_h, \hat v_h - v_h}_{\rd\T_h}   \\
   &\quad&+ s\an{\ngrad v_h, \hat u_h - u_h}_{\rd\T_h}  \nonumber \\
   &\quad& +  \an{ \tau \Proj_{k-1} (\hat u_h - u_h), \Proj_{k-1}(\hat v_h -  v_h)}_{\rd\T_h}.\nonumber
\end{eqnarray}
\subsection{Local conservativity}
 Let $K$ be an element of $\T_h$, and let $\chi_K$ denote a characteristic function on $K$.
 Taking $\bm v_h = \{\chi_K,0\}$ in \eqref{hdg-ws1}, we find that our method as well as the other HDG methods satisfies the followingt local conservation property
\begin{equation} \label{loccons}
     -\int_{\dK} \hat{\bm\sigma}(\bm u_h) \cdot \bm n ds
      = \int_K f dx,
\end{equation}
where  $\hat{\bm \sigma}$ is a numerical flux defined by
\[
     \hat{\bm \sigma}(\bm u_h) = \nabla u_h + \tau (\hat u_h- u_h)\bm n.
\]

\subsection{Implementation using the Gaussian quadrature formula} 
 In this section, we will show that the reduced stabilization term can be easily calculated by means of the Gaussian quadrature formula in the two-dimensional case.
 We can also avoid the calculation of the $L^2$ projections in the reduced stabilization term by using it.
  
 For simplicity, we consider the case of the interval $I = [-1,1]$.
 Let $\varphi_m$ be the Legendre polynomial of order $m \ge 0$ on $I$.
 Let $f$ be a smooth function on $I$.
 The $k$-point Gauss-Legendre quadrature rule   on $I$ is given by
\[
    \G_{k}[f] = \sum_{i=1}^k w_i f(a_i),
\]
where $\{a_i, w_i\}_{i=1}^k$ are the quadrature points and weights. 
 The standard stabilization term for $P_k$--$P_k$ approximation can be exactly computed by using the $(k+1)$-point Gauss-Legendre quadrature rule.
 If we use the $k$-point quadrature rule instead of $(k+1)$-point one, then the reduced stabilization term is obtained.
 
\begin{lemma} \label{lem:rbi}
Let $\Proj_{k-1}$ denote the $L^2$-projection from $L^2(I)$ onto 
$\mathcal P^{k-1}(I)$. Then we have,  for all $\hat u_h, \hat v_h \in \mathcal P^{k}(I)$,
\begin{equation} \label{eq:rbi1}
  \G_{k}[ \hat u_h \hat v_h] = \int_I \Proj_{k-1}\hat u_h \Proj_{k-1}\hat v_h ds.
\end{equation}
\end{lemma}

\begin{proof}
We can write $\hat u_h = \sum_{j=1}^k u_j \varphi_j$ 
and  $\hat v_h = \sum_{j=1}^k v_j \varphi_j$.
Note that
 the Legendre polynomial $\varphi_k$ vanishes at the quadrature points, i.e.,
  $\varphi_k(a_i) = 0$ for $1 \le i \le k$,
and that $\G_{k}$ is exact for polynomials of degree $\le 2k-1$. Then we have
\begin{eqnarray*}
  \G_{k}[ \hat u_h \hat v_h] 
  &=& \sum_{i=1}^k \left[ w_i \sum_{j=1}^k u_j\varphi_j(a_i) \cdot \sum_{j=1}^k v_j\varphi_j(a_i)\right] \\
  &=& \sum_{i=1}^{k}\left[w_i \sum_{j=1}^{k-1} u_j\varphi_j(a_i) \cdot \sum_{j=1}^{k-1} v_j\varphi_j(a_i)\right]\\
  &=& \sum_{i=1}^{k}\left[ w_i\Proj_{k-1} u_h(a_i) \Proj_{k-1} v_h(a_i)\right] \\
  &=& \G_{k}[\Proj_{k-1} \hat u_h \Proj_{k-1} \hat v_h] \\
  &=& \int_I \Proj_{k-1} \hat u_h \Proj_{k-1} \hat v_h ds,
\end{eqnarray*}
which completes the proof. \qed
\end{proof}

\paragraph{Remark.}
 In the three-dimensional case, the efficient implementaion using the Gaussian cubature formula is impossible since there exists almost no cubature formula such that the nodes are the common zeros of orthogonal polynomials (see, for example, \cite{LC94,CMS01}). 
 Even for a triangle, it is known that there does not exist such a cubature formula of degree $\ge 3$, see \cite{CS03}.
 Only in the case of $k=1$, our method can be easily implemented by the barycentric rule.  
  
\subsection{Relation with the Crouzeix-Raviart nonconforming finite element method}
  In \cite{CGSS09}, it is proved that the numerical trace of the EDG method coincides  with the approximate solution given by the conforming finite element method on a skeleton $\Gamma_h$.
 In this section, we reveal the relation between the Crouzeix-Raviart nonconforming finite element and our symmetric scheme($s=1$) with $P_1$--$P_0$ triangular elements.
 The meshes considered here are assumed to be triangular. 
 Let $\Pi_h$ denote the Crouzeix-Raviart interpolation operator with respect to a mesh $\T_h$. 
 For $\hat u \in L^2(\Gamma_h)$, the interpolation  $\Pi_h \hat u \in \mathcal P^1(\T_h)$ is given by
\[
     \int_{\Gamma_h} (\Pi_h \hat u) \hat v_h ds = \int_{\Gamma_h} \hat u \hat v_h ds 
     \quad \forall \hat v_h \in \mathcal P^0(\T_h).  
\]
\begin{theorem}
  Let $\bm u_h = \pair{u_h, \hat u_h} \in \bm V_h^{1,0}$ be the approximate solution  provided by \eqref{hdg-ws1} with $s=1$ and $u_{\rm CR}$ be the Crouzeix-Raviart approximation. Then we have
\begin{equation} \label{eq:cr0}
   \Pi_h \hat u_h = u_{\rm CR}.
\end{equation}
 In particular, we have for all $e \in \E_h$,
\begin{equation} \label{eq:cr1}
        \int_e \hat u_h ds = \int_e u_{\rm CR} ds. 
\end{equation}
\end{theorem}
\begin{proof}
By the definition of the Crouzeix-Raviart interpolation, we have
\[
     \an{  \ngrad{u_h},\Pi_h\hat v_h - \hat v_h}_{\rd\T_h} = 0.
\]
and
\begin{equation} \label{eq:stab}
   \an{\tau \Proj_{k-1}(\hat u_h - u_h), \Proj_{k-1}(\Pi_h\hat v_h - \hat v_h)}_{\rd\T_h} = 0.
\end{equation}
Taking $\bm v_h = \pair{\Pi_h \hat v_h, \hat v_h} \in \bm V_h^{1,0}$ in \eqref{hdg-ws1} yields
\begin{eqnarray}
     B_h(\bm u_h, \bm v_h)
     &=&  (\nabla u_h, \nabla (\Pi_h \hat v_h))_{\T_h}
     + s\an{\ngrad (\Pi_h \hat v_h), \hat u_h - u_h}_{\rd\T_h}  \label{eq:green} \\
   &=&  (1-s)(\nabla u_h, \nabla (\Pi_h v_h))_{\T_h}+  s( \nabla(\Pi_h \hat u_h), \nabla (\Pi_h \hat v_h))_{\T_h} \nonumber
\end{eqnarray}
  When $s=1$, the resulting equation for $\hat u_h$ reads
\begin{equation} \label{eq:hat}
  ( \nabla(\Pi_h \hat u_h), \nabla (\Pi_h \hat v_h))_{\T_h} =  (f, \Pi_h \hat v_h)_\Omega \quad \forall \hat v_h \in \mathcal P^0(\E_h). 
\end{equation}
  The solution of the equation above is uniquely determined to be $u_{\rm CR}$. 
  Hence we have $\Pi_h \hat u_h = u_{\rm CR}$.
  \qed
\end{proof}  
  \paragraph{Remark.}
  For higher-order polynomials, \eqref{eq:green} does not hold since the Laplacian
   of $u_h$ does not vanish.
  In the case of a polygonal element, the reduced stabilization term does not vanish, that is, \eqref{eq:stab} does not hold.
  Therefore we might not find a discrete equation in terms of only the hybrid unknown, like \eqref{eq:hat},
 in general cases.

\section{Error analysis}
 First, we prove the consistency, boundedness and coercivity of the bilinear form of our method.

\begin{lemma}[Consistency]
 Let $u$ be the exact solution of \eqref{poissoneq}\eqref{poissonbc},
 and $\bm u =\pair{u, u|_{\Gamma_h}}$.
 Then we have 
\begin{equation} \label{cons}
     B_h(\bm u, \bm v_h) = (f,v_h)_\Omega \quad \forall \bm v_h \in \bm V_h^{k,k}.
\end{equation}
\end{lemma}
 
\begin{proof}
 Since $\hat u - u = 0$ on $\Gamma_h$ and the normal derivative of $u$ is single-valued, we have
\begin{eqnarray}
     B_h(\bm u, \bm v_h) 
     &=& (\nabla u, \nabla v_h)_{\T_h} - \an{\ngrad{u}, v_h}_{\rd\T_h}\\ \nonumber
     &=& (-\Delta u, v_h)_{\T_h} \\ \nonumber
     &=& (f,v_h)_\Omega.
\end{eqnarray}
  \qed
\end{proof}

\begin{lemma}[Boundedness]
  There exists a constant $C_b$ independent of $h$ such that 
\begin{equation}  \label{bdd}
       |B_h({\bm w},  {\bm v})| \le C_b \enorm{\bm w} \enorm{\bm v}
   \quad \forall {\bm w}, {\bm v} \in \bm V(h)+\bm V_h^{k,k-1}.
\end{equation}
\end{lemma}  
  
\begin{proof}
 Let ${\bm w} =  \{w, \hat w\} = \{\bar w+w_h,  \bar w|_{\Gamma_h}+\hat w_h \}$ and
 $\bar{\bm v} =  \{v, \hat v\} = \{\bar v+v_h, \bar v|_{\Gamma_h} + \hat v_h\}$, 
 where $\bar w, \bar v \in H^2(\Omega)$ and $\{v_h, \hat v_h\}, 
 \{w_h, \hat w_h\} \in \bm V_h^{k,k-1}$.
 We estimate each term in the bilinear form separately.
 By the Schwarz inequality, we have
\begin{equation} \label{bdd1}
   |(\nabla w, \nabla v)_K| \le \|\nabla w\|_{0,K} \| \nabla v\|_{0,K}.
\end{equation}
To bound the second term in the bilinear form,
 we decompose it as      
\begin{align}
   \an{\ngrad{w}, \hat v - v}_{\rd\T_h} 
&  = \an{\ngrad{(\bar w + w_h)}, (\bar v + \hat v_h) - (
\bar v+v_h)}_{\rd\T_h} \notag \\ 
&= \an{\ngrad{\bar w},  \hat v_h - v_h}_{\rd\T_h} 
   +\an{\ngrad{w_h}, \hat v_h - v_h}_{\rd\T_h}.  
 \label{bdd2}
\end{align}
 Since $\an{\ngrad{\bar w}, z}_{\rd\T_h} = 0$ for any single-valued function $z$, we have
\begin{align} 
 &\an{\ngrad{\bar w},  \hat v_h - v_h}_{\rd\T_h} \notag \\ 
 &  =\an{\ngrad{\bar w}, \Proj_{k-1}(\hat v_h - v_h)}_{\rd\T_h} 
   +\an{\ngrad{\bar w}, (\I-\Proj_{k-1})(\hat v_h - v_h)}_{\rd\T_h} \notag \\ 
 & =\an{\ngrad{\bar w}, \Proj_{k-1}(\hat v_h - v_h)}_{\rd\T_h}  
   +\an{\ngrad{\bar w}, (\I-\Proj_{k-1})(\bar v - v_h)}_{\rd\T_h}.  
\label{bdd3}
\end{align}
Similarly, noting that $\an{\ngrad{w_h}, (\I-\Proj_{k-1})z}_{\rd\T_h} = 0$ for any single-valued function $z$, we have
\begin{align}
 &\an{\ngrad{w_h}, \hat v_h - v_h}_{\rd\T_h} \notag \\
 &=\an{\ngrad{w_h}, \Proj_{k-1}(\hat v_h - v_h)}_{\rd\T_h}
  +\an{\ngrad{w_h}, (\I-\Proj_{k-1})(\hat v_h - v_h)}_{\rd\T_h} \notag \\
 &=\an{\ngrad{w_h}, \Proj_{k-1}(\hat v_h - v_h)}_{\rd\T_h}
  +\an{\ngrad{w_h}, (\I-\Proj_{k-1})(\bar v - v_h)}_{\rd\T_h}.
  \label{bdd4}
 \end{align}
From \eqref{bdd2}, \eqref{bdd3} and \eqref{bdd4}, it follows that
\begin{equation} \label{bdd5} 
   \an{\ngrad{w}, \hat v - v}_{\rd\T_h} 
 = \an{\ngrad{w}, \Proj_{k-1}(\hat v_h - v_h)
 + (\I-\Proj_{k-1})(\bar v - v_h)}_{\rd\T_h}.
\end{equation}
 By the trace inequality and Lemma \ref{lem:hdgnorm}, we get
\begin{align}
 |\an{\ngrad{w}, \hat v - v}_{\rd\T_h}| 
 &  \le  C(|w|_{1,h}^2 + h^2|w|_{2,h}^2)^{1/2} 
         ( |{\bm v}|_{\rm j}^2 + |\bar v - v_h|_{1,h}^2)^{1/2} \notag \\
 & \le C\enorm{{\bm w}}\enorm{{\bm v}}.
\label{bdd6}
\end{align}
 In tha same manner, the third term in the bilinear form can be bounded. 
 The stabilization term is bounded as 
\begin{align}
 \left| \an{ \tau  \Proj_{k-1} (\hat w-  w), 
   \Proj_{k-1} (\hat v-v)}_{\rd\T_h}\right| 
 &= \left| \an{ \tau  \Proj_{k-1} ( \hat w_h- w_h), 
    \Proj_{k-1} (\hat v_h - v_h)}_{\rd\T_h}\right| \notag \\
 &\le  \tau_1|\bm w|_{\textrm j} |\bm v|_{\textrm j}. 
 \label{bdd7}
\end{align}
 Combining \eqref{bdd1}, \eqref{bdd6} and \eqref{bdd7}, we obtain 
\begin{eqnarray}
    |B_h(\bm w, \bm v)| \le C_b \enorm{\bm w} \enorm{\bm v},
\end{eqnarray}
 where the constant $C_b$  depends on the constants of the trace inequality and $\tau_1$, but is independent of $h$.  The proof is completed. \qed
\end{proof} 
\begin{lemma}[Coercivity] Assume that $\tau_0$ is sufficiently large. Then
 there exists a constant $C_c > 0$ independent of $h$ such that
\begin{equation} \label{coercivity}
  B_h(\bm v_h, \bm v_h) 
   \ge C_c \enorm{\bm v_h}^2 \qquad \forall \bm v_h \in \bm V_h^{k,k-1}.
\end{equation}
  When $s=-1$, it holds for any $\tau>0$.
\end{lemma}
\begin{proof}
  Letting $\bm u_h = \bm v_h$ in \eqref{hdg-ws1}, we have
\begin{eqnarray} \label{coer1}
   B_h(\bm v_h, \bm v_h) 
     &\ge& |v_h|^2_{1,h} -|1-s|\left|\an{\ngrad{v_h}, \hat v_h-v_h}_{\rd\T_h}\right| 
         +\tau_0|\bm v_h|_{\textrm j}^2.
\end{eqnarray}
 Note that
\begin{eqnarray}
  \an{\ngrad{v_h}, \hat v_h-v_h }_{\rd\T_h} 
    &=& \an{ \ngrad{v_h}, \Proj_{k-1}(\hat v_h-v_h)}_{\rd\T_h}  \label{coer03}.
\end{eqnarray}
 By the trace inequality and Young's inequality, it follows that, for any $\varepsilon>0$,
\begin{eqnarray}
  B_h(\bm v_h, \bm v_h) &\ge& (1-C\varepsilon) |v_h|^2_{1,h} 
       +(\tau_0-\varepsilon^{-1})|\bm v_h|_{\textrm j}^2.
\end{eqnarray}
 If $\tau_0 > C+1$, then we can take $\varepsilon =(\tau_0^{-1}+C^{-1})/2$.
 Therefore we obtain
\begin{eqnarray} \label{coer4}
     B_h(\bm v_h, \bm v_h) 
     &\ge&  \frac{1}{2} (|v_h|^2_{1,h} +  |\bm v_h|_{\textrm j}^2) \\ 
     &\ge& C \enorm{\bm v_h}^2, \nonumber
\end{eqnarray}
 where we have used the inverse inequality.
 If $s=-1$, then the second term in the right-hand side in \eqref{coer1} vanishes.
 From this,  we see that \eqref{coer4} holds for any $\tau > 0$ when $s=-1$.
  \qed
\end{proof}
 Next, we prove the error estimates with respect to the energy norm.  
\begin{theorem}[Quasi-best approximation] \label{thm:h1err}
 Let $u$ be the exact solution of \eqref{poissoneq}\eqref{poissonbc} and $\bm u := \pair{u, u|_{\Gamma_h}} \in \bm V$.
  Let $\bm u_h \in \bm V_h^{k, k-1}$ be an approximate solution provided by our  method \eqref{hdg-ws1}.
 Then we have 
\begin{equation} \label{errest}
  \enorm{\bm u - \bm u_h} \le C \inf_{\bm v_h \in \bm V_h^{k,k-1}} \enorm{\bm u - \bm v_h},  
\end{equation}
   where $C$ is a positive constant independent of $h$.
\end{theorem}
\begin{proof} 
 Let $\bm v_h \in \bm V_h^{k,k-1}$ be arbitrary.
 By the coercivity, consistency and boundedness, we have
\begin{eqnarray} \label{erresti01}
   C_c \enorm{\bm u_h - \bm v_h}^2 
    &\le& B_h(\bm u_h - \bm v_h, \bm u_h - \bm v_h) \\
    & = & B_h(\bm u - \bm v_h, \bm u_h - \bm v_h)  \nonumber \\   
    &\le& C_b \enorm{\bm u - \bm v_h}  \enorm{\bm u_h - \bm v_h}, \nonumber 
\end{eqnarray}
 from which it follows that
\begin{equation}
    \enorm{\bm u_h - \bm v_h} \le \frac{C_b}{C_c}\enorm{\bm u - \bm v_h}.
\end{equation} 
 By the triangle inequality, we have
\begin{eqnarray*}
     \enorm{\bm u - \bm u_h}
     &\le& \enorm{\bm u - \bm v_h} 
     + \enorm{\bm v_h - \bm u_h} \\
     &\le&  \left(1+ \frac{C_b}{C_c}\right)\enorm{\bm u - \bm v_h},
\end{eqnarray*}
 which implies \eqref{errest}. \qed
\end{proof}
 By the approximation property, the optimal-order error estimate in the energy norm follows immediately.
\begin{theorem} \label{thm:h1err2} 
  Let the notation be the same in Theorem \ref{thm:h1err}.
 If $u \in H^{k+1}(\Omega)$, then we have
\[
  \enorm{\bm u - \bm u_h} \le C h^k|u|_{k+1}.   
\]
\end{theorem}
 Finally, we prove the $L^2$-error estimates.
\begin{theorem}[$L^2$-error estimates] \label{thm:l2err}
 Let the notation be the same as in Theorem \ref{thm:h1err}.
 If $u \in H^{k+1}(\Omega)$, then we have 
\begin{eqnarray} \label{eq:l2err}
      \| u - u_h \|_0 \le C h^{k+1}|u|_{k+1} \quad \textrm{ for }  s=1,\\
       \label{eq:l2err-b}
      \| u - u_h \|_0 \le C h^{k}|u|_{k+1}  \quad \textrm{ for } s \neq 1,
\end{eqnarray}
 where $C$ is a positive constant independent of $h$.
\end{theorem}
\begin{proof}
 First, we prove \eqref{eq:l2err}.
 We can use Aubin-Nitsche's trick for $s=1$.
 Let $\psi \in H^2(\Omega) \cap H_0^1(\Omega)$ be the exact solution of the equation $-\Delta \psi = u - u_h$, and define $\bm\psi = \pair{\psi, \psi|_{\Gamma_h}}$. 
 For any $\bm \psi_h \in \bm V^{k,k-1}_h$, by the consistency and boundedness, we have
\begin{eqnarray}
  \|u-u_h\|_0^2 
   &=& B_h(\bm u - \bm u_h, \bm \psi) \\ \nonumber  
   &=& B_h(\bm u - \bm u_h, \bm \psi - \bm\psi_h)  \\ \nonumber  
   &\le& C_b \enorm{\bm u - \bm u_h}\enorm{\bm \psi - \bm\psi_h}. 
\end{eqnarray}
  Note that there exists $\bm \psi_h \in \bm V^{k,k-1}_h$ such that 
\[
    \enorm{\bm\psi - \bm\psi_h} \le Ch|\psi|_2 \le Ch\|u-u_h\|_0.
\]
 By Lemma \ref{thm:h1err2}, we  obatin \eqref{eq:l2err}.
 For the proof of the case $s \neq 1$, we show the following inequality in a similar manner presented in \cite{Arnold82}. 
\begin{equation} \label{eq:l2err2}
   \| w - w_h \|_0 \le C \enorm{\bm w - \bm w_h} 
    \quad \forall \bm w = \pair{w, w|_{\Gamma_h}} \in \bm V(h),
      \bm w_h \in \bm V_h^{k,k-1}. 
\end{equation}
 Let $\varphi \in H^2(\Omega) \cap H_0^1(\Omega)$ be the solution of $-\Delta \varphi = w-w_h$ and define $\bm\varphi=\pair{\varphi, \varphi|_{\Gamma_h}}$.  
 Then we have
\begin{eqnarray*}
    \|w - w_h\|_0^2 &=& (\nabla \varphi, \nabla (w-w_h))_{\T_h}
      - \an{ \bm n \cdot \nabla \varphi, \hat w_h-w_h }_{\rd\T_h} \\
      &\le& C\|\varphi\|_{2,\Omega} \enorm{\bm w - \bm w_h}.
\end{eqnarray*}
 Since $\|\varphi\|_{2,\Omega} \le C\|w-w_h\|_0$, we have \eqref{eq:l2err2}.
 From Theorem \ref{thm:h1err2},  \eqref{eq:l2err-b}  follows immediately. \qed
\end{proof}

\section{Numerical results}
\subsection{Discontinuous approximation}
 We consider the following test problem:
\begin{eqnarray}
      -\Delta u &=& 2\pi^2 \sin(\pi x) \sin(\pi y) \quad \textrm{ in } \Omega,\\
      u &=& 0 \quad \textrm{ on } \partial\Omega,
\end{eqnarray}
where the domain $\Omega$ is the unit square and the source function is chosen so that the exact solution is $u(x,y) = \sin(\pi x)\sin(\pi y)$.
 We employed unstructured triangular meshes and $P_k$--$P_{k-1}$ discontinuous approximation for $ 1\le k \le 3$. 
 The schemes are all symmetric. 
 Tables \ref{r-hdg-discont} and \ref{std-hdg-discont}  display the convergence histories  of the reduced and standard HDG schemes, respectively.
 The mesh size is given by $h \approx 0.1 \times 2^{-(l-1)}$.
 It can be observed that the convergence rates of the piecewise $H^1$-error and $L^2$-error are optimal in all the cases, which agrees with our theoretical results.
 We also see that the absolute errors of the reduced HDG method is approximately as same as those of the standard HDG method.  
 It suggests that the complementary projection part of the hybrid quantity, namely $(\I - \Proj_{k-1}) \hat u_h$, actually does not contribute to accuracy.
     
\begin{table}[h]
\caption{Convergence history of the reduced HDG methods with discontinuous approximations}
\centering
\begin{tabular}{ccccccc} \hline
     & & \multicolumn{2}{c}{$\| u - u_h\|$} & & 
         \multicolumn{2}{c}{$\| \nabla u - \nabla u_h\|$}   \\  \cline{3-4}\cline{6-7}
     &$l$& Error      & Order & & Error      & Order \\ \hline
P1P0 & 1 & 6.7399E-03 &  --   & & 2.7656E-01 & -- \\
     & 2 & 1.5971E-03 & 2.48  & & 1.3137E-01 & 1.28 \\
     & 3 & 3.9742E-04 & 2.05  & & 6.7522E-02 & 0.98 \\
     & 4 & 9.7854E-05 & 2.01  & & 3.2294E-02 & 1.06 \\ \hline
P2P1 & 1 & 1.2851E-04 & --    & & 2.1856E-02 & -- \\
     & 2 & 1.3557E-05 & 3.88  & & 4.8848E-03 & 2.58 \\
     & 3 & 1.5594E-06 & 3.18  & & 1.1857E-03 & 2.08 \\
     & 4 & 1.8789E-07 & 3.03  & & 2.8859E-04 & 2.03 \\ \hline
P3P2 & 1 & 5.7044E-06 & --    & & 1.1525E-03 & -- \\
     & 2 & 2.7034E-07 & 5.25  & & 1.1943E-04 & 3.91 \\
     & 3 & 1.8682E-08 & 3.93  & & 1.6192E-05 & 2.94 \\
     & 4 & 9.7700E-10 & 4.23  & & 1.7701E-06 & 3.17 \\ \hline
\end{tabular}
\label{r-hdg-discont}
\end{table}

\begin{table}[h]
\caption{Convergence history of the standard HDG methods with discontinuous approximations}
\centering
\begin{tabular}{ccccccc} \hline
 & & \multicolumn{2}{c}{$\| u - u_h\|$} & &  
   \multicolumn{2}{c}{$\| \nabla u - \nabla u_h\|$}   
\\   \cline{3-4}  \cline{6-7}
     &$l$& Error      & Order & & Error      & Order  \\ \hline
P1P1 & 1 & 4.5794E-03 & --    & & 2.5585E-01 & -- \\
     & 2 & 1.0083E-03 & 2.61  & & 1.2030E-01 & 1.30  \\
     & 3 & 2.7100E-04 & 1.93  & & 6.2282E-02 & 0.97  \\
     & 4 & 6.0270E-05 & 2.15  & & 2.9429E-02 & 1.07  \\ \hline
P2P2 & 1 & 1.2353E-04 & --    & & 2.2591E-02 & -- \\
     & 2 & 1.2762E-05 & 3.91  & & 5.0321E-03 & 2.59  \\
     & 3 & 1.5147E-06 & 3.14  & & 1.2195E-03 & 2.09  \\
     & 4 & 1.8201E-07 & 3.04  & & 2.9701E-04 & 2.02  \\ \hline
P3P3 & 1 & 5.8623E-06 & --    & & 1.1410E-03 & -- \\
     & 2 & 2.7867E-07 & 5.25  & & 1.1733E-04 & 3.92  \\
     & 3 & 1.9385E-08 & 3.92  & & 1.5902E-05 & 2.94  \\
     & 4 & 1.0046E-09 & 4.24  & & 1.7365E-06 & 3.17  \\ \hline
\end{tabular}
\label{std-hdg-discont}
\end{table}

\subsection{Continuous approximation} 
 As mentioned in Sect 2.6, the approximation property does not hold for the continuous approximations. As a result, the convergence order in the energy and $L^2$ norms 
 may not be optimal. We carried out numerical experiments 
 to observe the convergence rate.
  The same test problem as in the previous is considered. 
 We computed the approximate solutions by the reduced HDG method with the $P_2$--$P_1$ and $P_3$--$P_2$ continuous approximations.
 The same meshes as in the previous were used.
 The results are shown at Table \ref{r-hdg-cont}.
 We observe that the convergence rates in the piecewise $H^1$ and $L^2$ norms are sub-optimal, which indicates the reduced stabilization is not suitable for the continuous approximations.

\begin{table}[h]
\caption{Convergence history of the reduced HDG methods with continuous approximations}
\centering
\begin{tabular}{ccccccc} \hline
  & & \multicolumn{2}{c}{$\| u - u_h\|$} & & 
     \multicolumn{2}{c}{$\| \nabla u - \nabla u_h\|$} \\\cline{3-4}  \cline{6-7}   
     & $l$ & Error & Order   &  & Error & Order     \\  
     \hline
P2P1 & 1 & 6.6188E-03 & --   &  & 2.3451E-01 & --   \\
     & 2 & 1.4856E-03 & 2.58 &  & 1.1206E-01 & 1.27 \\
     & 3 & 3.7717E-04 & 2.02 &  & 5.6537E-02 & 1.01 \\
     & 4 & 8.9346E-05 & 2.06 &  & 2.7466E-02 & 1.04 \\
     \hline
P3P2 & 1 & 1.5101E-04 & --   &  & 1.4131E-02 & --   \\
     & 2 & 1.5681E-05 & 3.90 &  & 3.1251E-03 & 2.60 \\
     & 3 & 1.8543E-06 & 3.14 &  & 7.6405E-04 & 2.07 \\
     & 4 & 2.2391E-07 & 3.03 &  & 1.8754E-04 & 2.01 \\
     \hline
\end{tabular}
\label{r-hdg-cont}
\end{table}

\section{Conclusion}
 We proposed a new hybridized discontinuous Galerkin method with reduced stabilization. 
 We devised an efficient implementation of our method by means of the Gaussian quadrature  formula in the two-dimensional case.
 The error estimates in the energy and $L^2$ norms were proved  under the chunkiness condition.
 It was also shown that our method with $P_1$--$P_0$ approximation is closely related to the Crouzeix-Raviart  nonconforming finite element method.
 Numerical results confirmed the validity of the proposed schemes.
   
\section{Acknowledgement}
  This work was supported by JSPS KAKENHI Grant Number 26800089.
  \nocite{Ciarlet78}
  \nocite{KA03}
  \nocite{PW05}
  \nocite{Oikawa10}
  \nocite{KIO09}
  \nocite{ABCM02}

  \bibliographystyle{spmpsci}
  \bibliography{ref}  
\end{document}